\input{style/arxiv-ba.cfg}
\documentclass[ba,linksfromyear,preprint]{imsart}
\makeatletter
   \@ifpackageloaded{natbib}{}{\usepackage{natbib}}
\makeatother
\usepackage{xspace}

\pubyear{2015}
\volume{10}
\issue{2}
\firstpage{479}
\lastpage{499}
\doi{10.1214/15-BA942}

\makeatletter

\newcommand{\eg}{{e.g.\/}\xspace}

\newcommand{\todo}[1]{{\bf *** #1 ***}}

\newcommand{\half}{\frac{1}{2}}

\newcommand{\secref}[1]{\mbox{\S$\,$\ref{sec:#1}}}
\newcommand{\lemref}[1]{\mbox{Lemma~\ref{lem:#1}}}

\newcommand{\remref}[1]{\mbox{Remark~\ref{rem:#1}}}

\newcommand{\bd}{{\rm bd}}

\newcommand{\transp}{^{\prime}}
\newcommand{\cd}{\,|\,}

\newcommand{\rank}{{\rm rank}}
\newcommand{\cov}{{\rm Cov}}

\newcommand{\E}{{\mbox{E}}}

\newcommand{\R}{{\cal R}}

\newtheorem{theorem}{Theorem}[section]

\newtheorem{lem}[theorem]{Lemma}

\newtheorem{rem*}{Remark}[section]

\theoremstyle{definition}

\newcommand{\iid}{{independent and identically distributed}\xspace}

\renewcommand{\todo}[1]{}

\renewcommand{\bd}{{\bm d}} \newcommand{\bm}[1]{\mbox{\boldmath$#1$}}
\newcommand{\bmsub}[1]{\mbox{\boldmath{\scriptsize$#1$}}}
\newcommand{\bnabla}{{\bm\nabla}}
\newcommand{\bX}{{\bm X}}
\newcommand{\br}{{\bm r}} \newcommand{\bY}{{\bm Y}}
\newcommand{\bT}{{\bm T}} \newcommand{\bmu}{{\bm\mu}}
\newcommand{\bTheta}{{\bm\Theta}}
\newcommand{\Ybar}{\overline
Y}

\newcommand{\norm}{{\cal N}}
\renewcommand{\transp}{^{\mbox{\rm\scriptsize T}}}
\newcommand{\rss}{R}
\newcommand{\kl}{\mbox{\rm KL}}
\newcommand{\tr}{\mathop{\mathrm{tr}}}
\newcommand{\as}[1]{\mbox{$\mathop{\mathrm{a.\,s.}}\,\,[#1\,]$}}
\newcommand{\aic}{\mbox{\rm AIC}\xspace}

\newcommand{\var}{{\mbox{var}}}

\makeatother

\begin{document}

\begin{frontmatter}
\title{Bayesian Model Selection Based on Proper Scoring Rules\thanksref{T1}}
\runtitle{Bayesian Model Selection Based on Proper Scoring Rules}

\begin{aug}
\author{\fnms{A. Philip} \snm{Dawid}\corref{}\thanksref{t1}\ead[label=e1]{apd@statslab.cam.ac.uk}}
\and
\author{\fnms{Monica} \snm{Musio}\thanksref{t2}\ead[label=e2]{mmusio@unica.it}}

\runauthor{A.~P.~Dawid and M.~Musio}


\relateddois{T1}{Related articles:
DOI: \relateddoi[ms=BA942A]{Related item:}{10.1214/15-BA942A},
DOI: \relateddoi[ms=BA942B]{Related item:}{10.1214/15-BA942B},
DOI: \relateddoi[ms=BA942C]{Related item:}{10.1214/15-BA942C};
rejoinder at DOI: \relateddoi[ms=BA942REJ]{Related item:}{10.1214/15-BA942REJ}.}

\thankstext{t1}{University of Cambridge, {apd@statslab.cam.ac.uk}}
\thankstext{t2}{University of Cagliari, {mmusio@unica.it}}

\end{aug}

%
\begin{abstract}
Bayesian model selection with improper priors is not well-defined
because of the dependence of the marginal likelihood on the arbitrary
scaling constants of the within-model prior densities. We show how this
problem can be evaded by replacing marginal log-likelihood by a
homogeneous proper scoring rule, which is insensitive to the scaling
constants. Suitably applied, this will typically enable consistent
selection of the true model.
\end{abstract}

%


\end{frontmatter}


\section{Introduction}
\label{sec:intro}

The desire for an ``objective Bayesian'' approach to model selection
has produced a wide variety of suggested methods, none entirely
satisfactory from a principled perspective. Here we develop an
approach based on the general theory of proper scoring rules, and show
that, suitably deployed, it can evade problems associated with
arbitrary scaling constants, and deliver consistent model selection.

\section{Bayesian Model Selection}
\label{sec:modsel}

Let ${\cal M}$ be a finite or countable class of statistical models
for the same observable $\bX\in{\cal X} \subseteq\R^k$. Each
$M\in{\cal M}$ is a parametric family, with parameter $\theta_M \in
{\cal T}_M$, a $d_M$-dimensional Euclidean space; when $M$ obtains, with
parameter value $\theta_M$, then $\bX$ has distribution
$P_{\theta_M}$, with Lebesgue density $p_M( {\bm x}\cd\theta_M)$. Having
observed data $\bX= {\bm x}$, we wish to make inference about which model
$M\in{\cal M}$ (and possibly which parameter-value $\theta_M$)
actually generated these data.

A subjective Bayesian would begin by assigning a discrete prior
distribution over ${\cal M}$, with $\alpha(M)$, say, the assessed
probability that the true model is $M\in{\cal M}$; and, within each
model $M$, a prior distribution $\Pi_M$ for its parameter $\theta_M$
(to be interpreted as describing conditional uncertainty about
$\theta_M$, given the validity of model $M$). For simplicity we
suppose that $\Pi_M$ has a density function, $\pi_M(\theta_M)$, with
respect to Lebesgue measure $d\theta_M$ over ${\cal T}_M$.

The {\em predictive density function\/} of $\bX$, given only the
validity of model $M$, is
%
\begin{equation}
\label{eq:preddens}
p_M({\bm x}) = \int_{{\cal T}_M} p_M({\bm x}\cd\theta_M)\, \pi
_M(\theta_M)\,d\theta_M.
\end{equation}
This can be thought of as a hybrid between an ``objective'' component,
$p_M(x \cd\theta_M)$, and a ``subjective'' component,
$\pi(\theta_M)$.

Considered as a function of $M\in{\cal M}$, for given data ${\bm x}$,
$p_M({\bm x})$ given by \eqref{eq:preddens}---or any function on ${\cal M}$
proportional to this---supplies the {\em marginal likelihood\/}
function, $L(M)$, over $M\in{\cal M}$, based on data ${\bm x}$:
%
\begin{equation}
\label{eq:L}
L(M) \propto p_M({\bm x}).
\end{equation}

The posterior probability $\alpha(M \cd{\bm x})$ for model $M$ is then
given by Bayes's formula:
%
\begin{equation}
\label{eq:bayes}
\alpha(M \cd{\bm x}) \propto\alpha(M) \times L(M)
\end{equation}
where the omitted multiplicative constant is adjusted to ensure
$\sum_{M\in{\cal M}}\alpha(M \cd{\bm x}) = 1$. In particular, the
{\em
odds\/}, $\alpha(M_1)/\alpha(M_2)$, in favour of one model $M_1$
versus another model $M_2$, are multiplied, on observing $\bX={\bm
x}$, by
the {\em Bayes factor} $L(M_1)/L(M_2)$.

However, although the Bayes factor is ``objective'' to the extent that
it does not involve the initial discrete prior distribution $\alpha$
over the model space ${\cal M}$, it does still depend on the prior
densities $\pi_{M_1}$, $\pi_{M_2}$, within the models being
compared. As shown in \citet{Dawid:2011}, if the data are
independently generated from a distribution $Q$, the log-Bayes factor,
$\log L({M_1})/L(M_2)$, behaves asymptotically as $n\{K(Q, M_2) - K(Q,
{M_1})\} +O_p(n^\half)$ when $K(Q, M_2) > K(Q, {M_1})$, where $K(Q,M)$
denotes the minimum Kullback--Leibler divergence between $Q$ and a
distribution in $M$; while, if $Q$ lies both in ${M_1}$ and in $M_2$
(so that $K(Q, M_2) = K(Q, {M_1}) = 0$), with $q(x) \equiv p( x \mid
{M_1}, \theta_1^*) \equiv p(x \mid M_2, {\theta}_2^*)$ say, we have
log-Bayes factor
%
\begin{equation}
\label{eq:true}
\log\frac{L({M_1})}{L(M_2)} = \half(d_{M_2} - d_{M_1}) \log
\frac{n}{2\pi e} + \log
\frac{\rho(\theta_1^*\mid {M_1})}{\rho({\theta_2}^* \mid M_2)} +V,
\end{equation}
where $\rho(\theta\mid{M}) = {\pi_{M}(\theta)}/\{\det
I_M(\theta)\}^{\half}$ is the ``invariantised'' prior density with
respect to the Jeffreys measure on $M$; $V = O_p(1)$, with asymptotic
expectation $0$; and the dependence of $V$ on the prior specification
is $O_p(n^{-\half})$.

We thus see that, at any rate for comparing models of different
dimension, the dependence of the Bayes factor on the within-model
prior specifications is typically negligible compared with the leading
term in the asymptotic expansion. Nevertheless, many Bayesians have
agonised greatly about that dependence, and have attempted to
determine an ``objective'' version of the Bayes factor. The most
obvious approach, of using improper within-model priors, is plagued
with difficulties: the term $\rho(\theta^*\mid M)$ is perfectly
well-defined when we have a fully specified prior density, integrating
to 1; but when the prior density is non-integrable this function is
specified only up to an arbitrary scale factor---and \eqref{eq:true} will
depend on the chosen value of this factor. A~variety of {\em ad
hoc\/} methods have been suggested to evade this problem (see, for
example, \citet{OH,Ber-Per}). These methods are necessarily somewhat
subtle---one might even say contorted---and often do not even respect
the leading term asymptotics of \eqref{eq:true}.

In \citet{Dawid:2011}, it was argued that the problem of model
selection with improper priors can largely be overcome by focusing
directly on the posterior odds, rather than the Bayes factor, between
models. An alternative approach, that we develop here, is to replace
the Bayes factor by something different (but related), that is
insensitive to the scaling of the prior. For preliminary accounts of
this idea, see \citet{Musio-Dawid:2013,apd/mm:metron}.

\section{Proper Scoring Rules}
\label{sec:PSR}

The log-Bayes factor for comparing models ${M_1}$ and $M_2$ is
%
\begin{equation}
\label{eq:logBF}
\log p_{M_1}({\bm x}) - \log p_{M_2}({\bm x}).
\end{equation}

One way of interpreting \eqref{eq:logBF} is as a comparison of the {\em
log-scores\/} \citep{Good:1952} of the two predictive density
functions, $p_{M_1}(\cdot)$ and $p_{M_2}(\cdot)$, for $\bX$, in the
light of the observed data ${\bm x}$. That is, defining $S_L({\bm x},
Q) =
-\log q({\bm x})$, for any proposed distribution $Q$ with density function
$q(\cdot)$ over ${\cal X}$, and ${\bm x}\in{\cal X}$, we can interpret
the {\em log-score\/} $S_L({\bm x}, Q)$ as a measure of how badly $Q$ did
at forecasting the outcome ${\bm x}$; then the log-Bayes factor measures
by how much the log-score for ${M_1}$ (using the associated predictive
density) was better (smaller) than that for $M_2$.

Now the above definition of the log-score, $S_L({\bm x},Q)$, is just one
of many functions $S({\bm x},Q)$ having the property of being a {\em
proper scoring rule\/} (see, \eg\ \citet{Dawid:1986}): this is the
case if, defining $S(P,Q)$ as the expected score, $\E_{\bmsub{X} \sim
P}S(\bX,Q)$, when $\bX$ has distribution $P$, $S(P,Q)$ is minimised,
for any given $P$, by the ``honest'' choice $Q = P$. Associated with
any proper scoring rule is a {\em generalised entropy function\/}:
\begin{displaymath}
H(P) := S(P,P),
\end{displaymath}
and a non-negative {\em discrepancy function\/}:
\begin{displaymath}
D(P,Q) := S(P,Q) - H(P).
\end{displaymath}
These reduce to the familiar Shannon entropy and Kullback--Leibler
discrepancy when $S$ is the log-score.

\todo{New bit starts}

Standard statistical theory is largely based on the log-score
(corresponding to log-likelihood), the Shannon entropy, and the
Kullback--Leibler discrepancy. However, a very large part of that
theory generalises straightforwardly when these are replaced by some
other proper scoring rule, and its associated entropy and discrepancy:
see \citet{apd/mm/lv} for applications of proper scoring rules to
general estimation theory. Use of a proper scoring rule other than
the log-score typically sacrifices some efficiency for gains in
computational efficiency and/or robustness. Because there is a wide
variety of proper scoring rules, this offers greatly increased
flexibility. The choice of which specific rule to use may be based on
external considerations---for example, derived from the loss function
of a real decision problem \citep{pdg/apd:ams04}; or chosen for
convenience---for example, for reasons of tractability or robustness
\citep{apd/mm:metron}.

\todo{New bit ends}

In this paper we explore the implications and ramifications, for
Bayesian model selection, of replacing the log-score by some other
proper scoring rule as a yardstick for measuring and comparing the
quality of statistical models. In particular, we shall see that, for a
certain class of such proper scoring rules, the problems with improper
priors simply do not arise.

\section{Prequential Application}
\label{sec:preq}

Let $\bX= (X_1, X_2, \ldots)$, $\bX^n = (X_1, X_2, \ldots,X_n)$. Let
$Q$ be a distribution for $\bX$, with induced joint distribution
$Q^n$, having density $q^n(\cdot)$, for $\bX^n$. Using a prequential
(sequential predictive) approach \citep{dawi:1984}, decompose $q^n$
into its sequence of recursive conditionals:
%
\begin{equation}
\label{eq:q} q^n({\bm x}^n) = q_1(x_1) \times q_2(x_2)\times
\cdots\times q_n(x_n)
\end{equation}
where $q_i(\cdot)$ is the density function of the distribution $Q_i$ of
$X_i$, given $\bX^{i-1} = {\bm x}^{i-1}$; note that this depends on
${\bm x}^{i-1}$, even though the notation omits this. We now apply a
proper scoring rule $S_i$ (the form of which could in principle even
depend on ${\bm x}^{i-1}$) to the $i$th term in \eqref{eq:q}, and cumulate
the scores to obtain the {\em prequential score\/}
\[
S^n({\bm x}^n,Q) := \sum_{i=1}^n S_i(x_i, Q_i),
\]
where $Q_i$ is a function of ${\bm x}^{i-1}$. It is readily seen that
this yields a proper scoring rule for $\bX^n$ (strictly proper if
every $S_i$ is).

Define
%
\begin{equation}
\label{eq:delta}
\Delta^n({\bm x}^n;P,Q) := S^n({\bm x}^n,Q) - S^n({\bm x}^n,P),
\end{equation}
and
%
\begin{equation}
\label{eq:dn}
D^n({\bm x}^n;P,Q) := \sum_{i=1}^{n} D_i(P_i,Q_i),
\end{equation}
where $D_i$ is the discrepancy function associated with the component
scoring rule $S_i$. Then $D^n$ is in fact a function of ${\bm x}^{n-1}$.

Now $D^n\geq0$ is non-decreasing, and under suitable conditions we
will have $D^n \rightarrow\infty$ \as{P}. One useful condition for
this is the following:
\begin{lem}
\label{lem:kabanov}
Suppose that $P$ and $Q$ are mutually singular (as distributions for
the infinite sequence $\bX$), and for all $i$ and some $k>0$,
$D_i(P_i,Q_i) \geq kH^2(P_i,Q_i)$, where $H$ denotes Hellinger
distance. Then $D^n \rightarrow\infty$ \as{P}.
\end{lem}
\begin{proof}
Singularity implies $\sum_{i=1}^n H^2(P_i,Q_i) \rightarrow\infty$
\as{P} \citep{Kabanov:1977}.
\end{proof}
\begin{rem*}
\label{rem:kostas}
We can replace $H^2$ in \lemref{kabanov} by any other discrepancy
measure dominating (a multiple of) $H^2$, including
Kullback--Leibler divergence,
and $d_\epsilon$ given by $d_\epsilon(P,Q) = \int
|1-q(x)/p(x)|^\epsilon\,p(x)\,dx$ for $1\leq\epsilon\leq2$
\citep{Skouras:1998}. This latter is the $L_1$-distance for
$\epsilon=1$ and the squared $\chi^2$-distance for $\epsilon=2$.
\end{rem*}

Also,
%
\begin{equation}
\label{eq:mart}
U^n := \Delta^n(\bX^n;P,Q) - D^n(\bX^n;P,Q)
\end{equation}
is a $0$-mean martingale under $P$: indeed, it is the difference of the
two $0$-mean martingales
%
\begin{equation}
\label{eq:martq}
S^n(\bX^n,Q) - S^n(P^n,Q^n)
\end{equation}
and
%
\begin{equation}
\label{eq:martp}
S^n(\bX^n,P) - H^n(P^n).
\end{equation}
Under suitable and reasonable conditions on the behaviour of the
increments $S_i(x_i,Q_i) - S_i(x_i,P_i)$ of $\Delta_n(P,Q)$, $|U_n|$
will remain small in comparison with $D^n$. For example, if the
increments are all of similar size, a martingale law of the iterated
logarithm (see, \eg\citet{stout:1970}) would restrict $\sup_n|U_n|$
to have order $(n\log\log n)^\half$, while $D_n$ would be of order $n$.
It would then follow that, with $P$-probability~1, $\Delta^n
\rightarrow\infty$. In such a case, if $P$ is the true distribution
generating the data, then eventually we will have, with probability~1,
$S^n(\bX^n,P) < S^n(\bX^n,Q)$. Then choosing the model with the
lowest prequential score $S^n$ will yield a consistent criterion for
selecting among a finite collection of distributions for $\bX$.

\subsection{Application to Model Selection}
\label{sec:app} The above theory can be applied to the case that $P$,
$Q$ are the predictive distributions associated with different
Bayesian models, $M$ and $N$. In particular, suppose we have
statistical models
%
\begin{equation}
\label{eq:modp}
{\cal P} = \{P_{\theta}: \theta\in{\cal T}\}
\end{equation}
with prior $\Pi$ over ${\cal T}$; and
%
\begin{equation}
\label{eq:modq}
{\cal Q} = \{Q_\phi: \phi\in{\cal F}\}
\end{equation}
with prior $K$ over ${\cal F}$; and corresponding predictive distributions
%
\begin{eqnarray}
\label{eq:predp}
P &=& \int_{\cal T} P_\theta\, d\Pi(\theta),\\
Q &=& \int_{\cal F} Q_\phi\,dK(\phi).
\end{eqnarray}
Under conditions that allow application of the above results, we will
have $P(A)=1$, where $A$ is the event $S^n(\bX^n,Q) -
S^n(\bX^n,P)\rightarrow\infty$. Since $P(A) = \int_{\cal T}
P_\theta(A)\,d\Pi(\theta)$, we must have $P_\theta(A) = 1$ for
$\theta\in S$, where $\Pi(S) = 1$. In particular, if $\Pi$ has
Lebesgue density $\pi$ that is everywhere positive, then $P_\theta(A)
= 1$ for almost all $\theta\in{\cal T}$. So the criterion $S^n$ will
choose the correct model with probability 1 under (almost) any
distribution in that model. This result generalises the consistency
property of log-marginal likelihood \citep{Dawid:1992} to other proper
scoring rules.

\section{Local Scoring Rules}
\label{sec:local}

We call a scoring rule $S({\bm x},Q)$ {\em local (of order $m$)\/} if it
can be expressed as a function of ${\bm x}$, and of the density function
$q(\cdot)$ of $Q$ and its derivatives up to the $m$th order, all
evaluated at ${\bm x}$. Thus the log-score is local of order $0$. For
the case that the sample space ${\cal X}$ is an interval on the real
line, \citet{Parry:2012} have characterised all proper local scoring
rules. It was shown that these can all be expressed as a linear
combination of the log-score and a ``key local'' scoring rule, which
is a proper local scoring rule that is {\em homogeneous\/} in the
sense that its value is unchanged if $q$ and (thus) all of its
derivatives are multiplied by some constant $c>0$.

This property of a key local scoring rule has been found useful in
estimation theory. In standard likelihood inference, we need to
compute, and differentiate with the respect to the parameter, the
log-normalising constant of the statistical model distributions; and this
can be computationally prohibitive. But if, instead of log-score, we
use a key local scoring rule, the normalising constant simply does not
figure in the score, so simplifying computation: for some examples,
see \citet{apd/mm:asta,apd/mm:metron}. Applied to model selection,
this suggests a way of evading the problematic normalising constant of
the compleat Bayesian analysis: if we replace the log-score in
\eqref{eq:logBF} by some key local scoring rule, the dependence on the
normalising constant will disappear. Indeed, there is no problem in
computing such a score even for an ``improper'' density $q(\cdot)$,
having infinite integral over ${\cal X}$.

For any $k\geq1$, the simplest key local\footnote{Some conditions on
the behaviour of densities at the boundary of ${\cal X}$ are
required in order for \eqref{eq:genhy} to be a proper scoring rule.}
scoring rule is the order-$2$ rule of
\citet{Hyvarinen:2005}:\footnote{For convenience we have introduced an
extra factor of $2.$}
%
\begin{equation}
\label{eq:genhy}
S_H({\bm x},Q):=2\Delta\log q({\bm x})+\left\|\bnabla\log q({\bm
x})\right\|^{2},
\end{equation}
where $\bnabla$ denotes gradient, and $\Delta$ is the Laplacian
operator $\sum_{i=1}^k\partial^2/(\partial x_i)^2$. The associated
discrepancy function is
%
\begin{equation}
\label{eq:hyvdisc}
D_H(p,q) = \int\left\|\bnabla\log p({\bm x}) - \bnabla\log q({\bm
x})\right\|^2
p({\bm x})\,d{\bm x}.
\end{equation}

Variations on \eqref{eq:genhy} and \eqref{eq:hyvdisc} can be obtained, on
first performing a non-linear transformation of the space ${\cal X}$,
or equipping ${\cal X}$ with the structure of a Riemannian space and
reinterpreting $\bnabla$, $\Delta$ accordingly
\citep{Dawid-Lauritzen:2005}. Other key local scoring rules for the
multivariate case are considered by \citet{Parry:2013}. Though such
variations can be useful, here we largely confine ourselves to the
basic Hyv\"arinen score $S_H$ of \eqref{eq:genhy}. However, there remains
some freedom as to how this is applied: for example, we could apply
the multivariate score directly to the data, or to a sufficient
statistic, or cumulate the 1-dimensional scores associated with each
term in the decomposition \eqref{eq:q} \citep{vm/mm/apd:hy}. While such
manipulations have no effect on comparisons based on the log-score
$S_L$, they do typically affect those based on the Hyv\"arinen score
$S_H$. There is thus greater flexibility to apply this in useful
ways, \eg\ to ease computation, to improve robustness to model
misspecification, or (as in \xch{Section~\ref{sec:preq}}{\secref{preq}}) to ensure other desirable
properties such as consistency.

\section{Multivariate Normal Distribution}
\label{sec:mvn}

Consider in particular the case that the distribution $Q$ of $\bX$ is
multivariate normal:
%
\begin{equation}
\label{eq:mvn}
\bX\sim\norm_k(\bmu, \Sigma),
\end{equation}
with density
%
\begin{equation}
\label{eq:mvndens}
q({\bm x}) \propto\exp\{-\half({\bm x}-\bmu)^T \Phi({\bm x}-\bmu
)\}
\end{equation}
where $\Phi:= \Sigma^{-1}$ is the precision matrix, and (in contrast
to the usual convention for likelihood functions) the ``constants''
implicit in the proportionality sign are allowed to depend on the
parameters, $\bmu$ and $\Phi$, but not on ${\bm x}$.

We have
%
\begin{eqnarray}
\label{eq:mvnbnabla}
\bnabla\log q &=& - \Phi({\bm x}-\bmu),\\
\label{eq:mvndelta}
\Delta\log q &=& - \tr\Phi
\end{eqnarray}
so that, applying \eqref{eq:genhy},
%
\begin{equation}
\label{eq:mvnhyv}
S_H({\bm x},Q) = \left\|\Phi({\bm x}-\bmu)\right\|^2 - 2\tr\Phi.
\end{equation}
The associated discrepancy between $P = \norm_k(\bmu_P,\Phi_P^{-1})$
and $Q = \norm_k(\bmu_Q,\Phi_Q^{-1})$ is
%
\begin{equation}
\label{eq:mvndisc}
D_H(P,Q) = \tr\left(\Phi_P - 2\Phi_Q + \Phi_P^{-1}\Phi_Q^2\right)
+ \left\|\Phi_Q\left(\bmu_P-\bmu_Q\right)\right\|^2.
\end{equation}

The score \eqref{eq:mvnhyv} may be relatively easy to compute if the
model is defined in terms of its precision matrix $\Phi$, as for a
graphical model. Note also that, whereas the log-score $S_L$ in this
case would involve computing the determinant of $\Phi$, this is not
required for $S_H$.

We can now compare different hypothesised multivariate normal
distributions $Q$ for the observed data ${\bm x}$ by means of their
associated $S_H$ scores given by \eqref{eq:mvnhyv}.

\subsection{Univariate Case}
\label{sec:uninorm}

For the univariate case $Q = \norm(\mu,\sigma^2)$ we get
%
\begin{eqnarray}
\label{eq:unihyv}
S_H(x,Q) &=&
\frac1 {\sigma^4}\left\{(x-\mu)^2 - 2 \sigma^2\right\},\\
\label{eq:unidisc}
D_H(P,Q) &=&
\frac1 {\sigma_Q^4}
\left\{
\frac{\left(\sigma_P^2-\sigma_Q^2\right)^2}{\sigma_P^2}
+\left(\mu_P-\mu_Q\right)^2
\right\}.
\end{eqnarray}

In this case the Kullback--Leibler discrepancy is given by
%
\begin{equation}
\label{eq:unikl}
2\kl(P,Q) = \frac{\sigma_P^2}{\sigma_Q^2} + \log\frac{\sigma
_Q^2}{\sigma_P^2}
+ \frac{(\mu_P-\mu_Q)^2}{\sigma_Q^2}-1.
\end{equation}
Using $\log x \leq x-1$, we find
%
\begin{equation}
\label{eq:klhell}
D_H(P,Q) \geq\frac2 {\sigma_Q^2} \kl(P,Q).
\end{equation}

In the context of \xch{Section~\ref{sec:preq}}{\secref{preq}}, where $P$ and $Q$ are both Gaussian
processes for $(X_1, X_2, \ldots)$, we can apply \remref{kostas} to
deduce that prequential model comparison between $P$ and $Q$ based on
the Hyv\"arinen score will be consistent whenever $P$ and $Q$ are
mutually singular, and (writing $\sigma_{Q,i}^2$ for the variance,
under $Q$, of $X_i$, given $(X_1,\ldots,X_{i-1})$),
\[
\lim\inf_{i\rightarrow\infty}\sigma_{Q,i}^2>0\quad\as{P},
\]
and likewise with $P$ and $Q$ interchanged.

\section{Bayesian Model}
\label{sec:bayesmod}

For the Bayesian the parameter is a random variable, $\Theta$ say.
Let the statistical model have density
\begin{math}
p({\bm x}\mid\theta)
\end{math}
at $\bX={\bm x}$, when $\Theta=\theta$. If the prior density is
$\pi(\theta)$, the marginal density of ${\bm x}$ is
\begin{displaymath}
q({\bm x}) = \int p({\bm x}\mid\theta)\,\pi(\theta)\,d\theta.
\end{displaymath}
Then we find
\begin{eqnarray*}
\frac{\partial\log q({\bm x})}{\partial x_i} &=& \E\left\{\left
.\frac{\partial\log p({\bm x}\mid\Theta)}{\partial x_i}
\right| \bX= {\bm x}\right\},\\
\frac{\partial^2 \log q({\bm x})}{\partial x_i^2} &=& \E\left\{
\left.
\frac{\partial^2\log p({\bm x}\mid\Theta)}{\partial x_i^2} \right|
\bX= {\bm x}\right\}
+ \var\left\{\left.\frac{\partial\log p({\bm x}\mid\Theta
)}{\partial x_i}
\right| \bX= {\bm x}\right\}
\end{eqnarray*}
where the expectations and variances are taken under
the posterior distribution of $\Theta$ given $\bX= {\bm x}$, having
density $\pi(\theta\mid{\bm x}) = p({\bm x}\mid
\theta)\,\pi(\theta)/q({\bm x})$. This yields
%
\begin{eqnarray}
\label{eq:smix}
\nonumber
S_H({\bm x},Q) &=& \sum_i\left(\E\left[\left.
2\frac{\partial^2\log p({\bm x}\mid\Theta)}{\partial x_i^2}
+ 2\left\{\frac{\partial\log p({\bm x}\mid\Theta)}{\partial
x_i}\right\}^2
\right|
\bX= {\bm x}\right]\right.\\
&&{}\left.- \left[\E\left\{\left.\frac{\partial\log p({\bm
x}\mid\Theta)}{\partial x_i}
\right| \bX= {\bm x}\right\}\right]^2\right)\\
\nonumber
&=& \E\left\{\left.S_H({\bm x}, P_\Theta)\right| \bX={\bm
x}\right\}\\
&&{}+ \sum_i
\var\left\{\left.\frac{\partial\log p({\bm x}\mid\Theta
)}{\partial
x_i}\right| \bX={\bm x}\right\}.
\end{eqnarray}
%

\subsection{Exponential Family}
\label{sec:expfam}

Suppose further that the model is an exponential family with natural
statistic $\bT= {\bm t}({\bm x})$:
%
\begin{equation}
\label{eq:expgen}
\log p({\bm x}\mid{\bm
\theta}) = a({\bm x}) + b({\bm
\theta}) + \sum_{j=1}^k \theta_j t_j({\bm x}).
\end{equation}

Define $\bmu\equiv\bmu({\bm x})$, $\Sigma\equiv\Sigma({\bm x})$
to be the
posterior mean-vector and dispersion matrix of $\bTheta$, given
$\bX={\bm x}$. Then we obtain
\begin{displaymath}
S_H({\bm x},Q) = 2\Delta a + 2\bd\transp\bmu
+ \left\|\bnabla a + J\bmu\right\|^2 + 2\tr J\Sigma J \transp
\end{displaymath}
with $\bd\equiv\bd({\bm x}):= (\Delta t_j)$, $J\equiv J({\bm x}) :=
(\partial t_j({\bm x})/\partial x_i)$.

For the special case $\bT= \bX$, this becomes
\begin{displaymath}
S_H({\bm x},Q) = 2\Delta a
+ \left\|\bnabla a + \bmu\right\|^2 + 2\tr\Sigma.
\end{displaymath}
%

\section{Linear Model: Variance Known}
\label{sec:linmodknown}

Consider the following normal linear model for a data-vector $\bY=
(Y_1, \ldots, Y_n)\transp$:
%
\begin{equation}
\label{eq:lm}
\bY\sim\norm(X{\bm
\theta}, \sigma^2 I),
\end{equation}
where $X$ $(n \times p)$ is a known design matrix of rank $p$, and
${\bm
\theta}\in\R^p$ is an unknown parameter vector. In this section, we
take $\sigma^2$ as known.

\subsection{Multivariate Score}
\label{sec:joint}

Consider giving ${\bm
\theta}$ a normal prior distribution:
%
\begin{equation}
\label{eq:unif}
{\bm
\theta}\sim\norm(\bm{m},V).
\end{equation}

The marginal distribution $Q$ of $\bY$ is then
%
\begin{equation}
\label{eq:margy}
\bY\sim\norm(X\bm{m}, XVX\transp+ \sigma^2 I)
\end{equation}
with precision matrix
\begin{eqnarray*}
\Phi&=& (XVX\transp+ \sigma^2 I)^{-1}\\
&=& \sigma^{-2}\left\{I - X\left(X\transp X + \sigma^2 V^{-1}\right
)^{-1}X\transp\right\}
\end{eqnarray*}
on applying the Woodbury matrix inversion lemma ((10) of
\citet{Lindley-Smith:1972}).

An ``improper'' prior can now be generated by allowing $V^{-1}
\rightarrow0$, yielding
\begin{displaymath}
\Phi= \sigma^{-2} \Pi
\end{displaymath}
where
\begin{displaymath}
\Pi:= I - XAX\transp,
\end{displaymath}
with $A: = (X\transp X)^{-1}$, is the projection matrix onto the space
of residuals.

Although this $\Phi$ is singular, and thus cannot arise from any
genuine dispersion matrix, there is no problem in using it in
\eqref{eq:mvnhyv}. We obtain
%
\begin{equation}
\label{eq:normhyvimproper}
S_H({\bm y}, Q) = \frac1 {\sigma^4}\left(\rss- 2\nu\sigma^2\right)
\end{equation}
where $\rss$ is the usual residual sum-of-squares for model
\eqref{eq:lm}, on $\nu:= n-p$ degrees of freedom. Note that, unlike
marginal log-likelihood, this is well-defined, in spite of the fact
that we have not specified a ``normalising constant'' for the improper
prior density. This is, of course, a consequence of the homogeneity of
the Hyv\"arinen score $S_H$.

The above analysis is not, however, applicable if $\rank(X) < p$---in
particular, whenever $n<p$. Taking $V^{-1} \rightarrow0$ is
equivalent to using an improper prior density $\pi({\bm
\theta}) \equiv c$,
with $0<c<\infty$. When $X$ is of rank $p$, the integral formally
defining the marginal density of $\bY$ is finite for each ${\bm y}$ (even
though the resulting density is itself improper). However, when
$\rank(X)<p$ this integral is infinite at each ${\bm y}$, so that no
marginal joint density---even improper---can be defined.

Using the criterion \eqref{eq:normhyvimproper} for comparing different
normal linear models, all with the same known residual variance
$\sigma^2$, is equivalent to comparing them in terms of their
penalised scaled residual sum-of-squares, $(\rss/\sigma^2) +
2p$---which is just Akaike's \aic for this known-variance case.
(However, when $\sigma^2$ varies across models, the criterion
\eqref{eq:normhyvimproper} is no longer equivalent to \aic.)

Now it is well known that \aic is not a consistent model selection
criterion. As an example, consider the two models:
\begin{eqnarray*}
M_1 &:& Y_i \sim\norm(0,1),\\
M_2 &:& Y_i \sim\norm(\theta,1).
\end{eqnarray*}
Then, with $\Ybar$ denoting the sample mean $\sum_i Y_i/n$, we have
$\aic_1 = \sum_i Y_i^2$, $\aic_2 = \sum_i(Y_i-\Ybar)^2+2$, so that
$\aic_1-\aic_2 = n\Ybar^2-2$. When $M_1$ holds, this is distributed,
for any $n$, as $\chi^2_1 -2$, which has a non-zero probability of
being positive, and thus favouring the incorrect model $M_2$.

Hence the above approach does not seem an entirely satisfactory
solution to the model-selection problem.

\subsection{Prequential Score}
\label{sec:seq} In an attempt to restore consistent model selection,
we turn to the prequential approach.

In \eqref{eq:lm}, let ${\bm x}_i$ be the $i$th row of $X$, and $X^i$ the
matrix containing the first $i$ rows of~$X$. Assuming $X$ is of full
rank, then $X^i$ is of full rank if and only if $i \geq p$.

Define, for $i\geq p\,$:
%
\begin{eqnarray}
\label{eq:A}
A_i &:=& \left\{(X^i)\transp(X^i)\right\}^{-1},\\
\label{eq:theta}
\widehat{\bm
\theta}_i &:=& A_i (X^i)\transp\bY^i
\end{eqnarray}
and, for $i > p\,$:
%
\begin{eqnarray}
\label{eq:eta}
\eta_{i} &:=& {\bm x}_{i}\transp\widehat{\bm
\theta}_{i-1},\\
\label{eq:k}
k_{i}^2 &:=& 1+{\bm x}_{i}\transp A_{i-1} {\bm x}_{i} = (1-{\bm x}_i
\transp A_i {\bm x}_{i})^{-1},\\
\label{eq:z}
Z_{i} &:=& k_{i}^{-1}(Y_i - \eta_i)
\end{eqnarray}
(where the identity in \eqref{eq:k} follows from the Woodbury
lemma).\vadjust{\eject}

Then for the improper prior \eqref{eq:unif} with $V^{-1}\rightarrow0$,
the predictive distribution of $Y_i$, given $\bY^{i-1}$, is
%
\begin{equation}
\label{eq:predy}
Y_i \sim\norm(\eta_i, k_i^2\sigma^2)\quad(i > p).
\end{equation}
That is, in the predictive distribution the $(Z_i: i=p+1, \ldots,n)$
are \iid$\norm(0, \sigma^2)$ variables (which property also holds in
the sampling distribution, conditionally on ${\bm
\theta}$); moreover,
$\rss= \sum_{i=p+1}^n Z_i^2$.

Note that, under the model \eqref{eq:lm}, $\eta_i$ has expectation
${\bm x}_i\transp\theta$ and variance $k_i^2-1$. So the predictive
distribution \eqref{eq:predy} and the true distribution will be
asymptotically indistinguishable (the property of ``prequentially
consistent'' estimation---see \citet{dawi:1984}) if and only if
%
\begin{equation}
\label{eq:kk}
k_i^2\rightarrow1\mbox{ as }i \rightarrow\infty.
\end{equation}
This we henceforth assume, for any model under consideration.

For $i>p$, the incremental score \eqref{eq:unihyv} associated with
\eqref{eq:predy} is
%
\begin{equation}
\label{eq:increm}
S_i = \frac{T_i}{k_i^2\sigma^2}
\end{equation}
where
%
\begin{equation}
\label{eq:ti}
T_i := \frac{Z_i^2}{\sigma^2} -2.
\end{equation}
Under any distribution in the model, the $(T_i)$ are independent, with
%
\begin{eqnarray}
\label{eq:tmean}
\E(T_i) &=& -1,\\
\var(T_i) &=& 2.
\end{eqnarray}

As discussed in \xch{Section~\ref{sec:preq}}{\secref{preq}}, minimising the cumulative prequential
score
%
\begin{equation}
\label{eq:normpreq}
S^* := \sum_{i} S_i
\end{equation}
should typically yield consistent model choice. We investigate this
in more detail in \xch{Section~\ref{sec:preqcons}}{\secref{preqcons}} below.

Expression~\eqref{eq:increm} is only defined for an index $i$ exceeding
the dimensionality of the model. When comparing models of differing
dimensionalities, we should ensure the identical criterion is used for
each. We could just cumulate the $S_i$ over indices $i$ exceeding the
greatest model dimension, $p_{\max}$ say, but this risks losing
relevant information. To restore this, we might add to that sum the
multivariate score \eqref{eq:normhyvimproper} computed, for each model,
for the first $p_{\max}$ observations.

\subsection{Multivariate or Prequential?}
\label{sec:or} The multivariate score \eqref{eq:normhyvimproper} can be
expressed as the sum of rescaled incremental scores:
%
\begin{equation}
\label{eq:sgen}
S_H({\bm y}, Q) = \frac1 {\sigma^2} \sum_{i=p+1}^n T_i= \sum
_{i=p+1}^{n} k_i^2 S_i,
\end{equation}
and the scaling factor $k_i^2$ has been assumed to satisfy \eqref{eq:kk}.
It would thus seem that \eqref{eq:sgen} is asymptotically equivalent to
\eqref{eq:normpreq}, and thus that model selection by minimisation of the
multivariate score \eqref{eq:normhyvimproper} should be consistent for
model choice. However, we have seen that this is not the case.

Further analysis dispels this paradox. The difference between the
prequential and the multivariate score, up to time $n$, is
%
\begin{equation}
\label{eq:diff}
S^* - S_H = \frac1 {\sigma^2} \sum_{i=p}^n \left(\frac1
{k_i^2}-1\right) T_i.
\end{equation}
Under any distribution in the model, this has expectation
\[
\frac1{\sigma^2} \sum_{i} \left(1 - \frac1 {k_i^2}\right) = \frac
1{\sigma^2} \sum_{i} {\bm x}_i\transp A_i{\bm x}_i,
\]
and variance
\[
\frac2{\sigma^4} \sum_{i=1}^n \left({\bm x}_i\transp A_i{\bm x}_i
\right)^2.
\]

Suppose the $({\bm x}_i)$ look like a random sample from a $p$-variate
distribution, with $\E{\bm x}_i{\bm x}_i\transp= C$. Then, for large $i$,
\begin{displaymath}
\E\left(i{\bm x}_i\transp A_i{\bm x}_i\right) = \E\tr\left\{\left
(\sum_{j=1}^i{\bm x}_j
{\bm x}_j\transp/i\right)^{-1}{\bm x}_i {\bm x}_i\transp\right\}
\approx\tr C^{-1}C = p.
\end{displaymath}
So $1-1/k_i^2 \approx p/i$; in particular \eqref{eq:kk} holds. Then
$\E(S^* - S_H) \approx(p /\sigma^2) \sum_{i=p}^n i^{-1} \approx
p(\log n)/\sigma^2$. A~similar analysis shows $\var(S^* - S_H) <
\infty$. Thus, under the model, $S^* - S_H \approx p(\log
n)/\sigma^2$. So, contrary to first impressions, the difference
between the cumulative prequential score $S^*$ and the multivariate
score $S_{H}$ diverges to infinity (at a logarithmic rate) under any
true model.

\subsection{Prequentially Consistent Model Selection}
\label{sec:preqcons}

We now consider the asymptotic behaviour of the cumulative prequential
score $S^*$, given by \eqref{eq:normpreq}, when used to select between
two models, $M_1$ and $M_2$, both of the general form \eqref{eq:lm}, when
$M_1$ is true. Let these models have respective dimensions $p_1$ and
$p_2$, and variances $\sigma_1^2$ and $\sigma_2^2$. Let $Z_i$,
$k_i^2$, as defined above, refer to $M_1$, and denote the
corresponding quantities for $M_2$ by, respectively, $W_i$, $h_i^2$.
Let $S^*_1$, $S^*_2$ denote the cumulative prequential scores for
$M_1$, $M_2$, respectively. We assume conditions on the regressors, as
discussed above, under which
%
\begin{eqnarray}
\label{eq:approxk}
1-1/k_i^2 &\approx& p_1/i,\\
\label{eq:approxh}
1-1/h_i^2 &\approx& p_2/i.
\end{eqnarray}

Since the $(Y_i)$ are independent normal variables with variance
$\sigma_1^2$, and the $(Z_i)$ and $(W_i)$ are, in each case,
constructed from the $(Y_i)$ by an orthogonal linear transformation,
we will have
%
\begin{eqnarray}
\label{eq:zdist}
Z_i &\sim& \norm(0,\sigma_1^2)\quad\mbox{independently,}\\
\label{eq:wdist}
W_i &\sim& \norm(\nu_i,\sigma_1^2)\quad\mbox{independently,}
\end{eqnarray}
where the $(Z_i)$ have mean $0$ since $M_1$ is true, whereas the
$(\nu_i)$ may be non-zero.

Let $p = \max\{p_1, p_2\}$. Apart from a finite contribution from
some initial terms, the difference in prequential scores, up to time
$n$, is
%
\begin{equation}
\label{eq:preqdiff}
S^*_2 - S^*_1 = \frac1 {\sigma_2^2}
\sum\frac1{h_i^2}\left(\frac{W_i^2}{\sigma_2^2}-2\right)
- \frac1 {\sigma_1^2}
\sum\frac1{k_i^2}\left(\frac{Z_i^2}{\sigma_1^2}-2\right)
\end{equation}
where $\sum$ denotes $\sum_{i=p+1}^n$.

On account of \eqref{eq:zdist} and \eqref{eq:wdist}, this has expectation
%
\begin{equation}
\label{eq:expdiff}
\E(S^*_2 - S^*_1) = \frac1 {\sigma_2^4}\sum\frac{\nu_i^2}{h_i^2}
+ \frac{(\sigma_1^2-\sigma_2^2)^2}{\sigma_1^2\sigma_2^4}\sum\frac1{h_i^2}
+ \frac1{\sigma_1^2}\sum\left(\frac1 {k_i^2}-\frac1 {h_i^2}\right).
\end{equation}

We now consider various cases for $M_2$.

\subsubsection{$M_2$ true}
\label{sec:M2true}

If the true distribution also belongs to $M_2$ (as well as to $M_1$),
then we must have $\sigma_2^2 = \sigma_1^2=\sigma^2$ say, and $\nu_i
\equiv0$. Then \eqref{eq:expdiff} reduces to
%
\begin{equation}
\label{eq:truediff}
\E(S^*_2 - S^*_1) = \frac1{\sigma^2}\sum\left(\frac1
{k_i^2}-\frac1 {h_i^2}\right).
\end{equation}
On account of \eqref{eq:approxk} and \eqref{eq:approxh}, this behaves
asymptotically as $(p_2-p_1)(\log n)/\sigma^2 + o(\log n)$. Also, an
analysis similar to that in \xch{Section~\ref{sec:or}}{\secref{or}} shows that $\var(S^*_2 -
S^*_1)$ is bounded, so that
%
\begin{equation}
\label{eq:truediffp}
S^*_2 - S^*_1 = \frac{(p_2-p_1)\log n}{\sigma^2} + o_p(\log n).
\end{equation}
(Compare this with the behaviour of the log-Bayes factor in this case,
which, in line with \eqref{eq:true}, is asymptotic to $\half
(p_2-p_1)\log
n$ when the within-model priors are proper).

In particular, when comparing finitely many true models of different
dimensions, minimising the cumulative prequential score will
consistently favour the simplest true model, at rate $\propto\log n$.

We now consider cases where $M_2$ is false. For simplicity we confine
attention to the expected score.

\subsubsection{Wrong variance}
\label{sec:samemean}
Suppose first that $M_2$ has the wrong variance $\sigma_2^2 \neq
\sigma_1^2$. In this case the first term in \eqref{eq:expdiff} is
non-negative, the second is positive of order $n$, and the third term
is again of order $\log n$. The true model $M_1$ is thus favoured, at
rate $\propto n$---just as for the log-score in the case of proper
priors.

\subsubsection{Right variance, wrong mean}
\label{sec:wrongmean}
Suppose now $\sigma_2^2 = \sigma_1^2= \sigma^2$, but the
data-generating distribution does not have the mean-structure of
$M_2$. We note that the log-Bayes factor \eqref{eq:true} will tend to
infinity (almost surely), so selecting the true model $M_1$, if and
only if $\sum\nu_i^2 = \infty$.

In this case we have
%
\begin{equation}
\label{eq:expdiffwrongmean}
\E(S^*_2 - S^*_1) = \frac1 {\sigma^4}\sum\frac{\nu_i^2}{h_i^2}
+ \frac1{\sigma^2}\sum\left(\frac1 {k_i^2}-\frac1 {h_i^2}\right),
\end{equation}
where $\nu_i \not\equiv0$ and $h_i^2 \not\equiv k_i^2$.

The first term in \eqref{eq:expdiff} is non-negative, while the second
term behaves asymptotically as $(p_2-p_1)(\log n)/\sigma^2$. In
particular, if $p_2 > p_1$, then \eqref{eq:expdiff} increases at rate at
least $(p_2-p_1)(\log n)/\sigma^2$, so favouring the true model.

However, things are more delicate if $p_2 < p_1$. In this case, if
$\sum(\nu_i/h_i)^2$ increases sufficiently slowly --- specifically, at
rate less than $(p_1-p_2)\sigma^2(\log n)$ --- then the increased
simplicity of model $M_2$ more than compensates for the slight
inaccuracy in its mean-structure, leading to selection of the slightly
incorrect model $M_2$.

The case $p_2=p_1$ requires a still more delicate analysis, which we
shall not pursue here.

\paragraph{Example}
As an example, consider again the comparison of the models $M_1$ and
$M_2$ of \xch{Section~\ref{sec:joint}}{\secref{joint}}.

Under $M_1$, with $Y_i \sim\norm(0,1)$, we have $p_1 = 0$, $k_i^2
=1$, $Z_i = Y_i$. In this special case the cumulative prequential
score $S^*_1$ is identical to the multivariate score $S_{H,1}$.

For model $M_2$, with $Y_i \sim\norm(\theta,1)$ ($\theta\neq0$), we
have $p_2 = 1$, $h^2_i = i/(i-1)$, $W_i = \{(i-1)/i\})^\half(Y_i -
\overline Y_{i-1}) \sim\norm(0,1)$. Although $h_i^2 \rightarrow1$,
\begin{math}
S_2^* - S_{H,2}
\end{math}
has (under any distribution in $M_2$, and hence also under the simpler
model $M_1$) expectation $\sum_{i=1}^n i^{-1} \approx\log n$, and
bounded variance 2$\sum_{i=1}^n i^{-2} \approx\pi^2/3$. Since $S^*_1
\equiv S_{H,1}$, and we have seen that $S_{H,2} - S_{H,1}$ is bounded
in probability under $M_1$, $S^*_2 - S^*_1$ diverges to infinity (at
rate $\log n$) under $M_1$---so consistently selecting the correct
model $M_1$.

On the other hand, under $M_2$ we have
\begin{math}
S^*_2 = \sum_i(1-1/i) (W_i^2 -2) = -n + o_p(n),
\end{math}
while $S^*_1 = \sum_i (Y_i^2-2) = n(\theta^2-1) + o_p(n)$, so that
$S^*_2 - S^*_1 = -n\theta^2 + o_p(n)$, which thus diverges to
$-\infty$ (this time at rate $n$)---so now consistently selecting the
correct model~$M_2$.

In summary, although the multivariate score \eqref{eq:normhyvimproper} is
more straightforward to compute, if consistent model selection is
regarded as an important criterion then the prequential score is to be
preferred.

\section{Linear Model: Variance Unknown}
\label{sec:linmodunknown}

Now suppose we don't know $\sigma^2$ in \eqref{eq:lm}. With $\phi=
1/\sigma^2$, we have model density
%
\begin{equation}
\label{eq:moddens}
p({\bm y}\cd\theta, \phi) \propto\phi^{\half n}\exp-\frac\phi
2\left\{
R + ({\bm
\theta}- \widehat{\bm
\theta})\transp X\transp X ({\bm
\theta}-\widehat{\bm
\theta})
\right\}
\end{equation}
where $\rss= {\bm y}\transp\Pi{\bm y}$, with $\Pi= I-XAX\transp$,
is the
residual sum of squares, on $\nu= n-p$ degrees of freedom.

The standard improper prior for this model is $\pi({\bm
\theta}, \phi)
\propto\phi^{-1}$. Multiplying \eqref{eq:moddens} by this and
integrating over $({\bm
\theta},\phi)$ yields the (improper) joint
predictive density\footnote{For the integral formally defining this
density to be finite at each point we require $\rank(X)\geq p+1$.}
%
\begin{equation}
\label{eq:predh}
p({\bm y}) \propto\rss^{-\half\nu},
\end{equation}
with logarithm (up to a constant)
%
\begin{equation}
\label{eq:predlh}
l = {-\half\nu} \log\rss .
\end{equation}
Writing $\br= \Pi{\bm y}$ (the residual vector), we find
%
\begin{eqnarray}
\label{eq:l1}
\frac{\partial l}{\partial y_i} &=& - \frac{\nu r_i}\rss,\\
\label{eq:l2}
\frac{\partial^2 l}{\partial y_i^2} &=& \nu\left(\frac
{2r_i^2}{\rss^2} - \frac{\pi_{ii}} \rss\right),
\end{eqnarray}
and so (noting $\sum_i \pi_{ii} = \nu$) the multivariate score
\eqref{eq:genhy} is
%
\begin{equation}
\label{eq:mhyvh}
S_H = - \frac{(\nu-4)}{\widehat\sigma^2}
\end{equation}
where $\widehat\sigma^2 = \rss/\nu$ is the usual unbiased estimator of
$\sigma^2$. So long as at least one model under consideration has
$\nu> 4$ (a very reasonable requirement), choosing a model by
minimisation of the predictive score is thus equivalent to minimising
$J := {\widehat\sigma^2}/(\nu-4)$.

Again, this model selection criterion is typically inconsistent. Thus
consider the comparison between models $M_1$ and $M_2$ of
\xch{Section~\ref{sec:joint}}{\secref{joint}}, now extended to have unknown variance $\sigma^2$. We
have
%
\begin{eqnarray}
\label{eq:u1}
J_1 &=& \frac{(n-1)S^2 + n\Ybar^2}{n(n-4)},\\
\label{eq:u2}
J_2 &=& \frac{S^2}{(n-5)}
\end{eqnarray}
where $S^2 := \sum_{i=1}^n(Y_i-\Ybar)^2/(n-1)$ is a consistent
estimate of $\sigma^2$ under either model. Then $M_2$ is preferred if
$J_2 < J_1$, which holds when
%
\begin{equation}
\label{eq:u1u2}
\frac{n\Ybar^2}{\sigma^2} > \frac{2n-5}{(n-5)}\,\frac{S^2}{\sigma
^2} \approx2
\end{equation}
for large $n$. But, under $M_1$, ${n\Ybar^2}/{\sigma^2} \sim
\chi^2_1$, so that there is a positive probability of the inequality
\eqref{eq:u1u2} holding, so favouring the more complex model $M_2$.

\subsection{Prequential Score}
\label{sec:sequ} From \eqref{eq:predh}, as a function of $y_i$ the
predictive density of $Y_i$ given ${\bm y}^{i-1}$ (for $i > p$) is
%
\begin{equation}
\label{eq:student}
p(y_i \cd{\bm y}^{i-1}) \propto\rss_i^{-\half\nu_i}
= \left(\rss_{i-1} + z_i^2\right)^{-\half\nu_i}
\end{equation}
where $\rss_i$ is the residual sum-of-squares based on ${\bm y}^i$, on
$\nu_i := i-p$ degrees of freedom, and $z_i = k_{i}^{-1}(y_i -
\eta_i)$, as given by \eqref{eq:eta}--\eqref{eq:z}. Applying the univariate
case of \eqref{eq:genhy} now yields (for $i>p$) the incremental score:
%
\begin{eqnarray}
\label{eq:inch}
S_i &=& \frac{\nu_i\left\{\left(4+\nu_i\right)Z_i^2-2\rss
_i\right\}}
{k_i^2 \rss_i^2}\\
&=& \frac{\left(1+\frac4{\nu_i}\right)Z_i^2 - 2 s_i^2}
{k_i^2 s_i^4},
\end{eqnarray}
where $s^2_{i} := \rss_{i}/\nu_{i}$ is the residual mean square, based
on $\bY^i$, under the model. The prequential score is now obtained by
cumulating $S_i$ over $i$. Once again, under reasonable conditions
this can be expected to yield consistent model
selection.\footnote{Again, an additional contribution of the form of
\eqref{eq:mhyvh}, computed for an initial string of observations, could
be incorporated to ensure fair comparison between models of
different dimension.}

We investigate this consistency property further, for the special case
of comparing two true models of different dimensions $p_1 < p_2$. We
saw in \xch{Section~\ref{sec:preqcons}}{\secref{preqcons}} that in this case, when the variance
$\sigma^2$ is known (and under reasonable assumptions on the models)
the prequential Hyv\"arinen score prefers the simpler model over the
more complex model, at rate \mbox{$(p_2-p_1)(\log n)/\sigma^2$}.

We consider the asymptotic behaviour of $S^* := \sum_{i=p+1}^n S_i$
under a distribution in the model.\footnote{Our analysis is
indicative, rather than fully rigorous.} In this case the $(Z_i: i
> p)$ are \iid as $\norm(0,\sigma^2)$.

Writing $U_i := (Z_i^2/\sigma^2)-1$, so that $\E(U_i)=0$, $\E(U_i^2) =
2$, we have
\begin{equation}
\label{eq:zform}
k_i^2\sigma^2 S_{i} = \frac{\left(1+\frac4 {\nu_i}\right
)(U_i+1)-2(\overline U_i+1)}{(\overline U_i+1)^2}
\end{equation}
with $\overline U_i := \nu_i^{-1}\sum_{j=p+1}^i U_j$ (where $\nu_i =
i-p$). Now $\overline U_i = O_p(i^{-\half})$. Expanding
\eqref{eq:zform} as a power series in $\overline U_i$ gives
%
\begin{equation}
\label{eq:taylor}
k_i^2\sigma^2 S_{i} = \sum_{r=0}^\infty(-1)^r \overline U_i^r\left\{
(r+1)\left(1+ \frac4 {\nu_i}\right)(U_i+1) - 2
\right\}
\end{equation}
so that
%
\begin{eqnarray}
\label{eq:prob1}
k_i^2\sigma^2 S_{i} - (U_i-1) &=& \frac4 {\nu_i} + \frac{4U_i} {\nu
_i}\\
\label{eq:prob2}
&&{} - 2 \overline U_i \left(U_i + \frac4 {\nu_i} + \frac{4U_i}
{\nu_i}\right)\\
\label{eq:prob3}
&&{} + \overline U_i^2\left(1 + 3U_i + \frac{12}{\nu_i} + \frac
{12U_i}{\nu_i}\right)\\
&&{} + O_p(i^{-3/2}).
\end{eqnarray}

Noting
%
\begin{eqnarray}
\label{eq:uu1}
\E(\overline U_i^2) &=& 2/{\nu_i},\\
\label{eq:uu2}
\E(\overline U_i U_i) &=& 2/{\nu_i},\\
\label{eq:uu3}
\E(\overline U_i^2 U_i) &=& 8 /{\nu_i^2},
\end{eqnarray}
we compute
%
\begin{equation}
\E\left\{k_i^2\sigma^2 S_{i} - (U_i-1)\right\} = \frac2 {i} + O(i^{-3/2}),
\label{eq:taylorE}
\end{equation}
%
whence, on account of \eqref{eq:kk},
%
\begin{equation}
\label{eq:plusdiff}
\E\left(S^* - S_0^*\right) = 2(\log n)/\sigma^2 + O(1)
\end{equation}
where $S_0^* = \sum_{i=p+1}^n (U_i-1)/(k_i^2\sigma^2)$ is the
cumulative prequential score \eqref{eq:normpreq} for the submodel in
which the correct variance $\sigma^2$ is known.

In the remainder of this section, we argue that $S^* - S_0^*$ differs
from its expectation \eqref{eq:plusdiff} by $O_p\{(\log n)^\half\}$.
Computations have been executed and/or checked using the software {\em
Mathematica\/}.

On cumulating the term $\propto U_i/\nu_i$ in \eqref{eq:prob1} we obtain
variance $\propto\sum_{i=p+1}^\infty\nu_i^{-2}$, which is finite.
So this yields a contribution that is $O_p(1)$.

Consider now the term $ \propto\overline U_i U_i$ in \eqref{eq:prob2}.
We find $\var(\overline U_i U_i) = 4 /{\nu_i} + O(\nu_i^{-2})$, and
$\overline U_i U_i$ and $\overline U_j U_j$ are uncorrelated for
$i\neq j$. Hence on cumulating the term $\overline U_i U_i$ in
\eqref{eq:prob2} from $i=p+1$ to $n$ we get variance $\approx
\sum_{i=p+1}^n 4 / {\nu_i} \approx4\log n$. Thus the random
variation in this term contributes $O_p\{(\log n)^\half\}$ to $S^* -
S_0^*$.

There is also a term $\propto\overline U_i/\nu_i$ in \eqref{eq:prob2}.
Since $\overline U_i/\nu_i = O_p(i^{-3/2})$, its cumulative sum is
$O_p(n^{-\half})$.

Now consider \eqref{eq:prob3}. We look first at the term $\overline
U_i^2$. We compute $\var\{(\overline U_i)^2\} = {8}/\nu_i^2 +
48/\nu_i^{3} = \lambda_i$, say; and, for $i<j$,
\[
\cov\{(\overline U_i)^2, (\overline U_j)^2\} = \left(\frac{\nu
_i}{\nu_j}\right)^2\lambda_i.
\]
Hence
\begin{eqnarray*}
\var\left\{\sum_{i=p+1}^n (\overline U_i)^2\right\} &=& \sum
_{i=p+1}^n \lambda_i
+ 2\sum_{i=p+1}^n\sum_{j=i+1}^n \left(\frac{\nu_i}{\nu_j}\right
)^2\lambda_i\\
&\leq& 56\left\{\sum_{i=1}^\nu i^{-2} + 2\sum_{i=1}^\nu\sum
_{j=i+1}^\nu j^{-2}\right\}
\end{eqnarray*}
(with $\nu=n-p$), since $\lambda_i \leq56/\nu_i^2$. We have
$\sum_{i=1}^\infty i^{-2} < \infty$, and, for large $i$,
$\sum_{j=i+1}^\nu j^{-2} < \sum_{j=i+1}^\infty j^{-2} \approx i^{-1}$.
So $\var\{\sum_{i=p+1}^n (\overline U_i)^2\}$ is of order
$\log n$, and cumulating the term $\overline U_i^2$ in \eqref{eq:prob3}
again makes a contribution $O_p\{(\log n)^\half\}$ over and above its
expectation.

Now consider the term $U_i\overline U_i^2$ in \eqref{eq:prob3}. We have
%
\begin{equation}
\label{eq:varubar2ubar}
\var(U_i\overline U_i^2) = \frac{24}{\nu_i^2} + \frac{1024}{\nu
_i^3}+\frac{4928}{\nu_i^4}
\end{equation}
and, for $i<j$,
%
\begin{equation}
\label{eq:covubar2ubar}
\cov(U_i\overline U_i^2,U_j\overline U_j^2) = \frac{48(\nu_i+4)}{\nu
_i^2\nu_j^2}.
\end{equation}

By an argument similar to that for $\overline U_i^2$, we find that
cumulating the term $U_i\overline U_i^2$ in \eqref{eq:prob3} again makes
a contribution $O_p\{(\log n)^\half\}$ (over and above its
expectation).

Putting everything together, we have
%
\begin{equation}
\label{eq:finaldiff}
S^* - S_0^* = 2(\log n)/\sigma^2 + O_p\{(\log n)^\half\}.
\end{equation}

Now we have shown in \xch{Section~\ref{sec:preqcons}}{\secref{preqcons}} that, for comparing two true
models $M_1$ and $M_2$ with known variance $\sigma^2$ and respective
dimensions $p_1 < p_2$, under conditions on the behaviour of the
$({\bm x}_i)$, the difference in their cumulative prequential scores
$S_0^*$ behaves asymptotically as $(p_2-p_1)(\log n)/\sigma^2$.
Since, from \eqref{eq:finaldiff}, the difference between the scores for
the unknown and known variance cases is $2(\log n)/\sigma^2 + o_p(\log
n)$ for any model, the identical behaviour applies in the case that
the variance is unknown.

\section{Discussion}
\label{sec:disc} Replacement of the traditional log-score by a proper
scoring rule, applied to the predictive density, supplies a general
method for avoiding some of the difficulties associated with the use
of improper prior distributions for conducting Bayesian model
comparison and selection. In particular, use of a homogeneous scoring
rule, such as the Hyv\" arinen rule, supplies a method for taming the
otherwise wild behaviour associated with the arbitrariness of the
normalising constant of such a prior distribution. Moreover, when
applied prequentially, scoring rule based model selection will
typically lead to consistent selection of the true model: we have
argued for this property both in general terms and in the context of
normal linear models with known or unknown variance,\todo{New/revised
bit starts} with their usual improper priors.

While the literature on ``objective'' Bayesian model selection
contains some valuable discussion of general principles---see, for
example, \citet{bayarri:2012}---most of it focuses on explorations and
recommendations of appropriate priors, or classes of priors, or
relationships between priors, for use in specified circumstances or
for specified purposes. When those priors are improper, as is
commonly the case, further manipulations and distortions of the Bayes
factor are required to produce a well-defined procedure. Our approach
here makes no specific recommendations, leaving users free to apply
their most favoured prior distributions. Instead, we have introduced
a very general procedure, based on homogeneous proper scoring rules,
that allows the use of improper priors, however selected, without
needing to worry about the arbitrariness of their scaling constants.

There remains the issue of the choice of homogeneous proper scoring
rule. There are no clear theoretical grounds for preferring one over
another. Purely for simplicity, we have confined attention to the
most basic homogeneous rule, the Hyv\"arinen score, but similar
results can be expected for other homogeneous scoring rules. Further
theoretical and computational exploration and comparison of the
properties of the various methods is clearly required. Such
exploration might be extended to their performance in other contexts:
for example, issues of consistent model selection when the number of
parameters increases with the sample size
\citep{moreno:2010,johnson:2012}.


%

\begin{acknowledgement}

We thank the Editor, Associate Editor and referees for their helpful
feedback on a previous version of this article.

\end{acknowledgement}


\begin{thebibliography}{24}

\bibitem[\protect\citeauthoryear{Bayarri et~al.}{2012}]{bayarri:2012}
Bayarri, M.~J., Berger, J.~O., Forte, A., and Garc\'{\i}a-Donato, G. (2012).
``Criteria for {B}ayesian Model Choice with
Application to Variable Selection.''
\emph{The Annals of Statistics\/}, 40: 1550--1577.
\bid{doi={10.1214/12-AOS1013}, issn={0090-5364}, mr={3015035}}
\bptok{addids}%
\endbibitem

\bibitem[\protect\citeauthoryear{Berger and Pericchi}{1996}]{Ber-Per}
Berger, J.~O. and Pericchi, L.~R. (1996).
``The Intrinsic {B}ayes Factor for Model Selection and
Prediction.''
\emph{Journal of the American Statistical Association\/}, 91:
109--122.
\bid{doi={10.2307/2291387}, issn={0162-1459}, mr={1394065}}
\bptok{addids}%
\endbibitem

\bibitem[\protect\citeauthoryear{Dawid}{1984}]{dawi:1984}
Dawid, A.~P. (1984).
``{S}tatistical Theory---{T}he Prequential Approach (with
Discussion).''
\emph{Journal of the Royal Statistical Society, Series A\/}, 147:
278--292.
\bid{doi={10.2307/2981683}, issn={0035-9238}, mr={0763811}}
\bptok{addids}%
\endbibitem

\bibitem[\protect\citeauthoryear{Dawid}{1986}]{Dawid:1986}
--- (1986).
``Probability Forecasting.''
In: Kotz, S., Johnson, N.~L., and Read, C.~B. (eds.),
\emph{Encyclopedia of {S}tatistical {S}ciences\/}, volume~7, 210--218. New York:
Wiley-Interscience.
\bid{mr={0892738}}
\bptok{addids}%
\endbibitem

\bibitem[\protect\citeauthoryear{Dawid}{1992}]{Dawid:1992}
--- (1992).
``Prequential Analysis, Stochastic Complexity and {B}ayesian
Inference (with {D}iscussion).''
 In: Bernardo, J.~M., Berger, J.~O., Dawid, A.~P., and Smith,
A. F.~M.
(eds.), \emph{Bayesian {S}tatistics 4\/}, 109--125. Oxford: Oxford University
Press.
\bid{mr={1380273}}
\bptok{addids}%
\endbibitem

\bibitem[\protect\citeauthoryear{Dawid}{2011}]{Dawid:2011}
--- (2011).
``Posterior Model Probabilities.''
 In: Bandyopadhyay, P.~S. and Forster, M. (eds.), \emph{Philosophy of
{S}tatistics\/}, 607--630. New York: Elsevier.
\bptok{addids}%
\endbibitem

\bibitem[\protect\citeauthoryear{Dawid and Lauritzen}{2005}]{Dawid-Lauritzen:2005}
Dawid, A.~P. and Lauritzen, S.~L. (2005).
``The Geometry of Decision Theory.''
 In \emph{Proceedings of the Second International Symposium on
Information Geometry and its Applications\/}, 22--28. University of Tokyo.
12--16 December 2005.
\bptok{addids}%
\endbibitem

\bibitem[\protect\citeauthoryear{Dawid and Musio}{2013}]{apd/mm:asta}
Dawid, A.~P. and Musio, M. (2013).
``Estimation of Spatial Processes Using Local Scoring Rules.''
\emph{AStA Advances in Statistical Analysis\/}, 97: 173--179.
\bid{doi={10.1007/s10182-012-0191-8}, issn={1863-8171}, mr={3045766}}
\bptok{addids}%
\endbibitem

\bibitem[\protect\citeauthoryear{Dawid and Musio}{2014}]{apd/mm:metron}
--- (2014).
``Theory and Applications of Proper Scoring Rules.''
\emph{Metron\/}, 72: 169--183.
\bid{doi={10.1007/s40300-014-0039-y}, issn={0026-1424}, mr={3233147}}
\bptok{addids}%
\endbibitem

\bibitem[\protect\citeauthoryear{Dawid et~al.}{2015}]{apd/mm/lv}
Dawid, A.~P., Musio, M., and Ventura, L. (2015).
``Minimum Scoring Rule Inference.''
\textit{Scandinavian Journal of Statistics}, submitted for
publication.
\arxivurl{arXiv:1403.3920}
\bptok{addids}%
\endbibitem

\bibitem[\protect\citeauthoryear{Good}{1952}]{Good:1952}
Good, I.~J. (1952).
``Rational Decisions.''
\emph{Journal of the Royal Statistical Society, Series B\/}, 14:
107--114.
\bid{issn={0035-9246}, mr={0077033}}
\bptok{addids}%
\endbibitem

\bibitem[\protect\citeauthoryear{Gr\"unwald and Dawid}{2004}]{pdg/apd:ams04}
Gr\"unwald, P.~D. and Dawid, A.~P. (2004).
``Game Theory, Maximum Entropy, Minimum Discrepancy, and
Robust {B}ayesian Decision Theory.''
\emph{The Annals of Statistics\/}, 32: 1367--1433.
\bid{doi={10.1214/009053604000000553}, issn={0090-5364}, mr={2089128}}
\bptok{addids}%
\endbibitem

\bibitem[\protect\citeauthoryear{Hyv\"arinen}{2005}]{Hyvarinen:2005}
Hyv\"arinen, A. (2005).
``Estimation of Non-Normalized Statistical Models by Score
Matching.''
\emph{Journal of Machine Learning Research\/}, 6: 695--709.
\bid{issn={1532-4435}, mr={2249836}}
\bptok{addids}%
\endbibitem

\bibitem[\protect\citeauthoryear{Johnson and Rossell}{2012}]{johnson:2012}
Johnson, V.~E. and Rossell, D. (2012).
``{B}ayesian Model Selection in High-Dimensional Settings.''
\emph{Journal of the American Statistical Association\/}, 107:
649--660.
\bid{doi={10.1080/01621459.2012.682536}, issn={0162-1459}, mr={2980074}}
\bptok{addids}%
\endbibitem

\bibitem[\protect\citeauthoryear{Kabanov et~al.}{1977}]{Kabanov:1977}
Kabanov, Y.~M., Liptser, R.~S., and Shiryayev, A.~N. (1977).
``On the Question of Absolute Continuity and
Singularity of Probability Measures.''
\emph{\xch{Mathematics of the USSR. Sbornik}{Math. USSR-Sb}\/}, 33: 203--221.
\bptok{addids}%
\endbibitem

\bibitem[\protect\citeauthoryear{Lindley and Smith}{1972}]{Lindley-Smith:1972}
Lindley, D.~V. and Smith, A. F.~M. (1972).
``Bayes Estimates for the Linear Model (with {D}iscussion).''
\emph{Journal of the Royal Statistical Society, Series B\/}, 34:
1--41.
\bid{issn={0035-9246}, mr={0415861}}
\bptok{addids}%
\endbibitem

\bibitem[\protect\citeauthoryear{Mameli et~al.}{2014}]{vm/mm/apd:hy}
Mameli, V., Musio, M., and Dawid, A.~P. (2014).
``Comparisons of {H}yv\"arinen and Pairwise
Estimators in Two Simple Linear Time Series Models.''
\arxivurl{arXiv:1409.3690}
\bid{doi={10.1007/s40300-014-0039-y}, issn={0026-1424}, mr={3233147}}
\bptok{addids}%
\endbibitem

\bibitem[\protect\citeauthoryear{Moreno et~al.}{2010}]{moreno:2010}
Moreno, E., Gir\'on, F.~J., and Casella, G. (2010).
``Consistency of Objective {B}ayes Factors as the Model
Dimension Grows.''
\emph{The Annals of Statistics\/}, 38: 1937--1952.
\bid{doi={10.1214/09-AOS754}, issn={0090-5364}, mr={2676879}}
\bptok{addids}%
\endbibitem

\bibitem[\protect\citeauthoryear{Musio and Dawid}{2013}]{Musio-Dawid:2013}
Musio, M. and Dawid, A.~P. (2013).
``Local Scoring rules: A~Versatile Tool for Inference.''
 Paper presented at 59th World Statistics Congress, Hong Kong.
\href{http://www.statistics.gov.hk/wsc/STS019-P3-S.pdf}{http://www.}
\href{http://www.statistics.gov.hk/wsc/STS019-P3-S.pdf}{statistics.gov.hk/wsc/STS019-P3-S.pdf}
\bptok{addids}%
\endbibitem

\bibitem[\protect\citeauthoryear{O'Hagan}{1995}]{OH}
O'Hagan, A. (1995).
``Fractional {B}ayes Factors for Model Comparison.''
\emph{Journal of the Royal Statistical Society, Series~B\/}, 57:
99--138.
\bid{issn={0035-9246}, mr={1325379}}
\bptok{addids}%
\endbibitem

\bibitem[\protect\citeauthoryear{Parry}{2013}]{Parry:2013}
Parry, M.~F. (2013).
``Multidimensional Local Scoring Rules.''
Paper presented at 59th World Statistics Congress, Hong Kong.
\href{http://www.statistics.gov.hk/wsc/STS019-P2-S.pdf}{http://www.statistics.gov.hk/wsc/}
\href{http://www.statistics.gov.hk/wsc/STS019-P2-S.pdf}{STS019-P2-S.pdf}
\bptok{addids}%
\endbibitem

\bibitem[\protect\citeauthoryear{Parry et~al.}{2012}]{Parry:2012}
Parry, M.~F., Dawid, A.~P., and Lauritzen, S.~L. (2012).
``Proper Local Scoring Rules.''
\emph{The Annals of Statistics\/}, 40: 561--592.
\bid{doi={\\10.1214/12-AOS971}, issn={0090-5364}, mr={3014317}}
\bptok{addids}%
\endbibitem

\bibitem[\protect\citeauthoryear{Skouras}{1998}]{Skouras:1998}
Skouras, K. (1998).
``Absolute Continuity of {M}arkov Chains.''
\emph{Journal of Statistical Planning and Inference\/}, 75: 1--8.
\bid{doi={10.1016/\\S0378-3758(98)00117-7}, issn={0378-3758}, mr={1671674}}
\bptok{addids}%
\endbibitem

\bibitem[\protect\citeauthoryear{Stout}{1970}]{stout:1970}
Stout, W.~F. (1970).
``A~Martingale Analogue of {K}olmogorov's Law of the Iterated
Logarithm.''
\emph{Zeitschrift f\"ur Wahrscheinlichkeitstheorie und verwandte
Gebiete\/}, 15: 279--290.
\bid{mr={0293701}}
\bptok{addids}%
\endbibitem

\end{thebibliography}
\end{document}